\documentclass{article}

\usepackage[letterpaper,margin=1.2in]{geometry}

\usepackage{amsmath,amssymb,amsthm}
\usepackage{MnSymbol}
\usepackage{enumerate}
\usepackage{hyperref}
\usepackage{mathtools}
\usepackage{bbm}
\usepackage{xcolor}
\usepackage{tikz}
\usepackage{tikz-cd}

\usepackage{comment}

\newtheorem{introthm}{Theorem}

\newtheorem{introcor}[introthm]{Corollary}

\newtheorem{theorem}{Theorem}[section]

\newtheorem{lemma}[theorem]{Lemma}

\newtheorem*{theorem*}{Theorem}

\numberwithin{equation}{section}
\numberwithin{figure}{section}

\theoremstyle{definition}

\theoremstyle{remark}
\newtheorem{remark}[theorem]{Remark}

\newcommand\R{\mathbb{R}}
\newcommand\C{\mathbb{C}}

\DeclareMathOperator{\re}{Re}

\DeclareMathOperator{\Tr}{Tr}
\DeclareMathOperator{\tr}{tr}

\DeclareMathOperator{\sa}{sa}

\DeclareMathOperator{\Lip}{Lip}

\DeclarePairedDelimiter{\norm}{\lVert}{\rVert}
\DeclarePairedDelimiter{\ip}{\langle}{\rangle}

\title{A tail bound for multi-matrix models with Lipschitz interaction}
\author{%
	\setcounter{footnote}{1}%
	Charles-Philippe Diez\thanks{%
		Technion-Israel Institute of Technology, Faculty of Mathematics,
		Technion City, Haifa 3200003, Israel.\\
		\textbf{chphdiez@campus.technion.ac.il}.}\:
	and David Jekel\thanks{%
		Department of Mathematics and Statistics, University of Ottawa,
		150 Louis-Pasteur Private, Ottawa, ON K1N 6N5 (Canada).
		\textbf{djekel@uottawa.ca}.}
}

\begin{document}
	
	\maketitle
	\begin{abstract}
		We give non-asymptotic operator-norm tail bounds for random
		multi-matrix models with quadratic plus globally Lipschitz
		interactions. More precisely, let $\mu_f$ be the probability measure
		on $(\mathbb{M}_n)_{\sa}^m$ with density proportional to
		$\exp(-n^2(\norm{x}_2^2/2+f(x)))$. If $f$ is $L$-Lipschitz
		with respect to $\norm{\cdot}_1$, we construct a coupling of
		$X\sim\mu_f$ and an $m$-tuple $Z$ of independent
		GUE matrices such that
		\[
		\norm{X-Z}_\infty\leq L
		\]
		almost surely, where $\norm{\cdot}_\infty$ is the maximum of the operator norms of the matrix tuple. This is an operator-norm analog of the $L^\infty$-Wasserstein bound of Khudiakova, Maas, and Pedrotti (2025).  We give a new proof based on the variational formula of Bou{\'e} and Dupuis (1998),
		the associated Hamilton--Jacobi--Bellman equation, and the duality between entropy and pressure.
	\end{abstract}
	\section{Introduction}
	
	The goal of this work is to give useful tail bounds for the operator norms of multi-matrix models.  We consider a probability measure on the space $(\mathbb{M}_n)_{\sa}^m$ of $m$-tuples of self-adjoint complex matrices of the form
	\begin{equation} \label{eq: matrix measure}
		d\mu_f(x) = \frac{1}{C_f} \exp(-n^2(f(x) + \norm{x}_2^2 / 2)) \,dx, \qquad C_f = \int_{(\mathbb{M}_n)_{\sa}^m} \exp(-n^2(f(x) + \norm{x}_2^2 / 2)) \,dx,
	\end{equation}
	where $x = (x_1,\dots,x_m) \in (\mathbb{M}_n)_{\sa}^m$, $f: (\mathbb{M}_n)_{\sa}^m \to \R$ is some function for which the integral converges, and
	\[
	\norm{x}_2^2 = \sum_{j=1}^m \tr_n(x_j^2),
	\]
	for $\tr_n = (1/n) \Tr_n$ the normalized trace.  Our main motivation is the setting where $f$ is given by a dimension-independent formula such as
	\[
	f(x) = \tr_n(p(x))
	\]
	for a noncommutative polynomial $p$.  In this case, the integral expressing the constant $C_f$ is often called a matrix integral (see e.g.\ \cite{ParraudSchnelli2025} for additional history and motivation).  The large-$n$ limiting behavior of a random matrix tuple $X^{(n)}$ chosen according to the measure $\mu_f^{(n)} = \mu_f$ above will be described by a \emph{free Gibbs law} (see \cite{BS2001,GMS2006,GS2009}), provided the required limits as $n \to \infty$ exist.  The limiting behavior is well understood in the case $m = 1$ \cite{BdMPS} or the case that $f(x) + \norm{x}_2^2/2$ satisfies certain convexity conditions and other suitable hypotheses \cite{BS1998,GMS2006,GS2009,Dabrowski2017Laplace,JekelEntropy,JekelExpectation,DGS2021,JLS2022,ParraudSchnelli2025,GMS2026}, but there is still not much known for general choices of noncommutative polynomials or noncommutative smooth functions.  In fact, general results for the limiting behavior would go a long way to completing the large deviations programme of \cite{BCG2003} and the unification of free entropy \cite{Voiculescu2002}.
	
	The aim of this work is more modest.  Most analyses of random matrix models as in \eqref{eq: matrix measure} in the large-$n$ rely on estimates for the probability
	\begin{equation} \label{eq: tail probability}
		\mu_f(\{x: \norm{x}_\infty \geq R\}) \text{ where } \norm{x}_\infty = \max_j \norm{x_j},
	\end{equation}
	where $\norm{\cdot}$ denotes the operator norm, which allows one to truncate to operator norm balls when studying convergence of noncommutative polynomials in expectation and almost surely.  In the case where $f = 0$, $\mu_f$ becomes the GUE measure and it is well known that operator norm is close to $2$ with high probability.  Similar quite sharp bounds can be obtained in the case where $f(x) + \norm{x}_2^2/2$ is strongly convex, for instance, using log-concave convex correlation inequalities as in \cite[Proof of Theorem 3.4]{GMS2006}.  For the non-convex setting, essentially the only known result is due to Guionnet and Maurel-Segala, which proves operator norm bounds for polynomial potentials that satisfy the $(A,\eta,I)$-\emph{trapping} condition \cite[Theorem 2.2]{GMS2026} which in turn follows from the \emph{confining} condition \cite[Theorem 2.5]{GMS2026}; the proof uses combinatorial manipulations as well as changes-of-variables generated by flows.  The class of trapping potentials and the class of convex potentials intersect nontrivially, but neither is contained in the other.
	
	Rather than the technically challenging setting of higher-order polynomial potentials, we will prove an estimate for \eqref{eq: tail probability} for general globally Lipschitz functions $f$.  Although unfortunately $f(x) = \tr_n(p(x))$ for a noncommutative polynomial $p$ is usually not globally Lipschitz, the Lipschitz assumption still covers many explicit cases of interest.  First, in the $m = 1$ case, one can take $f(x) = \tr_n(\varphi(x))$ where $\varphi: \R \to \R$ is a sufficiently smooth function applied through functional calculus.  In fact, Lipschitz single-variable functions $\varphi$ were considered in \cite{MM2014transportentropy}, which showed that the associated free Gibbs laws will satisfy $L^1$-Talagrand transportation inequalities.  In the $m \geq 2$ setting, one can take $f$ to be the trace of a resolvent polynomial, i.e.,
	\[
	f(x) = \tr_n(p((x_{j_1}+\zeta_1)^{-1}, \dots, (x_{j_k} + \zeta_k)^{-1})),
	\]
	where $p$ is a non-commutative polynomial and $\zeta_j \in \C \setminus \R$.  The use of resolvent-polynomial test functions goes back at least to \cite{BCG2003} which expressed behavior of self-adjoint matrices $x_j$ in terms of their Cayley transforms $(x_j + i)(x_j - i)^{-1}$.  They have been used more recently in the context of free Gibbs laws in \cite{Dabrowski2017Laplace,JekelEntropy,JekelExpectation,GJNPviscositysolution} due to their conveniently tame behavior at $\infty$ compared to plain noncommutative polynomials.  Another similar construction is to take
	\[
	f(x) = \tr_n(p(\varphi(x_1),\dots,\varphi(x_m))),
	\]
	where $\varphi \in C_c^\infty(\R)$.  This fits into the more general framework(s) of noncommutative smooth functions developed in \cite{DGS2021,JLS2022}.  For proof that such $f$ are indeed Lipschitz, first note that $p$ is $\norm{\cdot}_1$-Lipschitz on each $\norm{\cdot}_\infty$ and that $x \mapsto \varphi(x)$ is globally $\norm{\cdot}_\infty$-bounded and $\norm{\cdot}_1$-Lipschitz by e.g.\ \cite[\S 4.2]{JLS2022}.
	
	More precisely, our estimates will be expressed in terms of the Lipschitz norm of $f$ with respect to the noncommutative $L^1$ norm on matrix tuples,
	\[
	\norm{x}_1 = \sum_{j=1}^m \tr_n((x_j^*x_j)^{1/2}).
	\]
	Here note that since $\norm{\cdot}_1$ is dual to the $\infty$ norm, $\norm{x}_\infty = \max_j \norm{x_j}$, the Lipschitz constant with respect to $\norm{\cdot}_1$ corresponds to an $L^\infty$ bound for $\norm{\nabla f}_\infty$.  We state our main result as a non-asymptotic bound that holds for $n \times n$ random matrix models.  We will use $\gamma$ to denote the measure of an $n \times n$ GUE tuple, namely,
	\begin{equation} \label{eq: GUE tuple measure}
		d\gamma(x) = \text{(constant)} e^{-n^2 \norm{x}_2^2/2}\,dx \text{ for } x \in (\mathbb{M}_n)_{\sa}^m.
	\end{equation}
	
	\begin{introthm} \label{thm: main}
		Let $f: (\mathbb{M}_n)_{\sa}^m \to \R$ be a function satisfying
		\[
		|f(x) - f(y)| \leq L \sum_{j=1}^m \norm{x_j - y_j}_1,
		\]
		and let $\mu_f$ be the probability measure on $(\mathbb{M}_n)_{\sa}^m$ given by
		\[
		d\mu_f(x) = \text{\emph{(constant)}} e^{-n^2(\norm{x}_2^2/2 + f(x))}\,dx,
		\]
		or $d\mu / d\gamma \propto e^{-n^2 f}$.  Then there exist random variables $X \sim \mu_f$ and $Z \sim \gamma$ with $\norm{X - Z}_\infty \leq L$.
	\end{introthm}
	
	Note that equality is achieved in the relation $\norm{X - Z}_\infty \leq L$ when we take $f(x) = L \sum_{j=1}^m \tr_n(x_j)$, since in this case $\mu_f$ is the translation of $\gamma$ by $L \cdot (1,\dots, 1)$.  We remark that the analogous result with the Euclidean norm $\norm{\cdot}_2$ rather than the operator norm $\norm{\cdot}_\infty$ in the conclusion is a special case of \cite[Theorem 2.1]{KhudiakovaMaasPedrotti2025}.  Namely, \cite{KhudiakovaMaasPedrotti2025} applies with the Gaussian measure replaced by a general strongly log-concave measure, but the conclusion bounds the $L^\infty$ Wasserstein distance specifically when the underlying norm is Euclidean.  However, we explain in Remark \ref{rem: alternative proof} that if the base measure is standard Gaussian, a version of the theorem holds for any norm on the underlying vector space.
	
	The following immediate corollary allows us to bound the tail probabilities for $X \sim \mu_f$ in terms of those for the GUE tuple $Z$.  We remark that the tail probability $\mathbb{P}( \norm{Z}_\infty \geq 2 + \delta)$ is well studied and there are many estimates for it; we include as a sample the Aubrun--Ledoux bound \cite{Ledoux2003remark,Aubrun2004} which has the correct small-deviation exponent.
	
	\begin{introcor}
		Let $f$ and $\mu_f$ be as in the previous theorem.  Let $X$ and $Z$ be random variables with distributions $\mu_f$ and $\gamma$ respectively.  Then for $\delta>0$,
		\begin{equation}
			\mathbb{P}(\norm{X}_\infty \geq L + 2 + \delta) \leq \mathbb{P}( \norm{Z}_\infty \geq 2 + \delta).
		\end{equation}
		The right-hand side in turn can be bounded by $(2m)C e^{-cn \delta^{3/2}}$ for some positive constants $C$ and $c$ by \cite{Ledoux2003remark,Aubrun2004}.
	\end{introcor}
	
	We arrived at Theorem \ref{thm: main} from looking at several other instances where Lipschitz assumptions lead to unexpectedly strong a priori operator norm bounds \cite[Theorem 7.18]{JLS2022}, \cite[Corollary 5.6]{JekelTypeCoupling}, and \cite[Lemma 5.8]{Jekel2026FIG}.  The conclusion of \cite[Theorem 7.18]{JLS2022} for free Gibbs laws is quite analogous to Theorem \ref{thm: main} but with a more complicated proof using iterative moment bounds based on integration by parts.
	
	Our proof of Theorem \ref{thm: main} is based on techniques in PDE theory and stochastic optimization.  The usage of such techniques in random matrix theory is not new; we emphasize in particular that stochastic variational formulas were central to the work of Biane--Capitaine--Guionnet \cite{BCG2003} and Dabrowski \cite{Dabrowski2017Laplace}.  We apply a formula of Bou{\'e} and Dupuis \cite{BoueDupuis1998variational} which shows that
	\[
	-\frac{1}{n^2}
	\log \int_{(\mathbb{M}_n)_{\sa}^m}
	\exp(-n^2 f)\,d\gamma
	=
	\inf_{\alpha}
	\mathbb{E}\left[
	f\left(Z_1+\int_0^1\alpha_t\,dt\right)
	+\frac{1}{2}\int_0^1\norm{\alpha_t}_2^2\,dt
	\right].
	\]
	where $(Z_t)_{t \in [0,1]}$ is a matrix Brownian motion and where $\alpha_t$ ranges over certain progressively measurable processes.  Standard analysis of Hamilton--Jacobi--Bellman equations allows us to obtain a minimizer $\alpha_t$ explicitly in terms of the gradient of the solution to a certain PDE, which ultimately means that $\norm{\alpha_t}_\infty \leq L$ where $L$ is the $\norm{\cdot}_1$-Lipschitz constant of $f$.  Moreover, it follows from the SDE and PDE that the distribution of $\int_0^1 \alpha_t\,dt + Z_1$ is exactly $\mu_f$, which produces the desired coupling between a Gaussian element $Z_1$ and the variable $X = \int_0^1 \alpha_t\,dt + Z_1$ for Theorem \ref{thm: main}.
	
	Remark \ref{rem: alternative proof} explains how the bound can also be proved using the stationary solution of the Langevin SDE, following a similar method as in \cite[Theorem~2.1, Step~1]{KhudiakovaMaasPedrotti2025}.  For us, the value of the variational approach is that it can also be applied in the free probabilistic context of \cite[\S 5]{Jekel2026FIG}, where there is not a straightforward analog of density, the Langevin-stationary distribution may not be unique in general (and existence requires substantial work), and so it is natural to \emph{define} Gibbs distributions using a variational framework.
	
	The tools applied here in the random matrix context originate from the theory of stochastic control,
	Wasserstein geometry, and functional inequalities for probability measures on $\R^m$ (and more generally, Riemannian manifolds) which has a long history.  On the optimal transport side, Talagrand~\cite{Talagrand1996}
	proved that the squared quadratic Wasserstein distance to the
	standard Gaussian measure is bounded by twice the relative
	entropy. Otto and Villani~\cite{OttoVillani2000} generalized these results for weighted manifolds under curvature assumptions and established
	the implication from logarithmic Sobolev inequalities to quadratic
	transportation cost inequalities (Talagrand's transportation cost inequalities). Bobkov, Gentil, and
	Ledoux~\cite{BobkovGentilLedoux2001} developed a Hamilton--Jacobi
	approach to this connection, using hypercontractivity and
	infimum convolution. We refer to Gozlan and
	L{\'e}onard~\cite{GozlanLeonard2010} for a survey of
	transportation-cost inequalities and their relationships
	with concentration and other functional inequalities.
	
	A stochastic approach to these questions is provided by
	variational representations obtained by Bou{\'e} and
	Dupuis~\cite{BoueDupuis1998variational} who established such
	a representation for logarithmic Laplace transforms of
	Brownian functionals, with a cost given by the quadratic
	energy of the drift. Borell~\cite{Borell2000} used in particular these stochastic
	control methods to prove geometric inequalities, including
	the Pr{\'e}kopa--Leindler inequality and therefore convexity principle.
	Lehec~\cite{Lehec2013} developed the corresponding representation
	of Gaussian relative entropy, and
	obtained short, unified proofs of several functional inequalities,
	including Talagrand's transportation-cost inequality and the
	logarithmic Sobolev inequality even in the infinite dimensional setting of Wiener space. His habilitation
	thesis~\cite{LehecHDR} gives a nice overview and detailed account of these
	stochastic methods.
	More recently, stochastic coupling arguments have also been
	used to prove sharp symmetrized transportation-cost inequalities
	and their geometric counterparts. Fathi~\cite{Fathi2018}
	established a sharp symmetrized form of Talagrand's
	transport-entropy inequality under a centering assumption,
	which is dual to the functional Blaschke--Santal{\'o} inequality.
	Courtade, Fathi, and
	Mikulincer~\cite{CourtadeFathiMikulincer2024} subsequently
	gave a stochastic proof of this inequality, using a coupling
	induced by time-reversed martingale representations.
	Lehec~\cite{Lehec2026Santalo} then simplified this argument
	to obtain a direct stochastic proof of the functional
	Blaschke--Santal{\'o} inequality in convex geometry.
	
	The duality between entropy and pressure used below is the
	classical Legendre duality between relative entropy and the
	logarithmic Laplace transform; see Gozlan and
	L{\'e}onard~\cite[Appendix~B, Proposition~B.1]{GozlanLeonard2010}
	and Lehec~\cite{Lehec2013}.
	For a broader perspective on the relationship between entropy
	minimization on path space and optimal transport through the
	Schr{\"o}dinger problem, see L{\'e}onard~\cite{Leonard2014}.
	
	\section{Preliminaries}
	
	We write $\mathbb{M}_n = M_n(\mathbb{C})$ for the space of $n \times n$ complex matrices, $(\mathbb{M}_n)_{\sa}$ for the subspace of self-adjoint elements, and $\tr_n$ for the normalized trace $(1/n) \Tr_n$.  Consistent with von Neumann algebraic conventions, we use Schatten-$p$ norms (or noncommutative $L^p$ norms) of tuples of matrices with respect to the normalized trace, namely,
	\[
	\norm{(x_1,\dots,x_m)}_p = \left( \sum_{j=1}^m \tr_n(|x_j|^p) \right)^{1/p} \text{ for } p \in [1,\infty),
	\]
	where $|x_j| = (x_j^* x_j)^{1/2}$ is defined by functional calculus.  Meanwhile,
	\[
	\norm{(x_1,\dots,x_m)}_\infty = \max_{j=1,\dots,m} \norm{x_j}_\infty,
	\]
	where $\norm{\cdot}_\infty$ is the operator norm.  We remark that $\tr_n(|x|^p)$ can also be expressed as $(1/n) \sum_{j=1}^n \sigma_j^p$ where $\sigma_1$, \dots, $\sigma_n$ are the singular values and $\norm{x}_\infty = \max_{j=1,\dots,n} \sigma_j$.
	
	We write $\ip{x,y}_2 = \tr_n(x^*y)$.  Recall that the noncommutative $L^p$ norms satisfy the analog of H{\"o}lder's inequality and duality of $L^p$ and $L^q$ when $1/p + 1/q = 1$, with the dual pairing given by $\tr_n(xy)$.  Since $\mathbb{M}_n$ is finite-dimensional, this duality holds even for $L^\infty$.  For background, see e.g.\ \cite{Simon2005} or \cite{PX2003}.
	
	\begin{lemma}\label{lem: Lipschitz gradient}
		Let $f:(\mathbb{M}_n)_{\sa}^m\to\R$ be differentiable, where the gradient is taken with respect to the inner product
		\[
		\ip{x,y}_2 = \sum_{j=1}^m \tr_n(x_jy_j).
		\]
		If $f$ is $L$-Lipschitz with respect to $\norm{\cdot}_1$, then
		$\norm{\nabla f(x)}_\infty\leq L$
		for every $x\in(\mathbb{M}_n)_{\sa}^m$. Moreover, if $f$ is $C^1$, then
		\[
		\Lip_{\norm{\cdot}_1}(f)
		=
		\sup_{x\in(\mathbb{M}_n)_{\sa}^m}
		\norm{\nabla f(x)}_\infty.
		\]
	\end{lemma}
	
	\begin{proof}
		We first recall that $\norm{\cdot}_1$ and $\norm{\cdot}_\infty$ are dual with respect to $\ip{\cdot,\cdot}_2$, that is 
		\[\underset{\norm{h}_1 \leq 1}{\sup}\lvert \langle a,h\rangle_2\rvert=\norm{a}_{\infty}
		\]
		\newline
		Now suppose that $f$ is $L$-Lipschitz with respect to $\norm{\cdot}_1$. For every $x,h\in(\mathbb{M}_n)_{\sa}^m$,
		\[
		\left|\ip{\nabla f(x),h}_2\right|
		=
		\left|\left.\frac{d}{dt}f(x+th)\right|_{t=0}\right|
		\leq
		L\norm{h}_1.
		\]
		Taking the supremum over $\norm{h}_1\leq1$ and using this duality yields
		\[
		\norm{\nabla f(x)}_\infty\leq L.
		\]
		In particular,
		\[
		\sup_{x\in(\mathbb{M}_n)_{\sa}^m}
		\norm{\nabla f(x)}_\infty
		\leq
		\Lip_{\norm{\cdot}_1}(f).
		\]
		For the converse, suppose that $f$ is $C^1$ and that
		\(
		K=\sup_{x\in(\mathbb{M}_n)_{\sa}^m}
		\norm{\nabla f(x)}_\infty<\infty.
		\)
		Then, for every $x,y\in(\mathbb{M}_n)_{\sa}^m$, the fundamental theorem of calculus and H\"older's inequality give
		\begin{align*}
			|f(y)-f(x)|
			&=
			\left|
			\int_0^1
			\ip{\nabla f(x+t(y-x)),y-x}_2\,dt
			\right|\\
			&\leq
			\int_0^1
			\norm{\nabla f(x+t(y-x))}_\infty
			\norm{y-x}_1\,dt\\
			&\leq
			K\norm{y-x}_1.
		\end{align*}
		Hence $\Lip_{\norm{\cdot}_1}(f)\leq K$, which, together with the first inequality, proves
		\[
		\Lip_{\norm{\cdot}_1}(f)
		=
		\sup_{x\in(\mathbb{M}_n)_{\sa}^m}
		\norm{\nabla f(x)}_\infty.
		\]
		If $K=\infty$, the first inequality already implies that
		$\Lip_{\norm{\cdot}_1}(f)=\infty$, so the identity holds in this case as well.
	\end{proof}
	
	Next, we recall some basic facts about relative entropy for probability measures, which is defined as
	\[
	h(\mu \mid \nu) = \int \frac{d\mu}{d\nu} \log \frac{d\mu}{d\nu} \,d\nu
	\]
	whenever $\mu \ll \nu$, and $\infty$ otherwise.
	For any two probability measures $\mu$ and $\nu$
	with $\mu \ll \nu$, writing $r=d\mu/d\nu$ gives the following representation (since $\mu,\nu$ are probability measures)
	\[
	h(\mu\mid\nu) = \int (r\log r-r+1)\,d\nu \geq 0,
	\]
	since $r\log r-r+1\geq 0,\: r\geq 0$ with equality if and only if $r=1$ $\nu$-almost everywhere,
	that is, $\mu=\nu$. If $\mu\not\ll\nu$, then
	$h(\mu\mid\nu)=\infty$ by definition.
	
	We will use a dimensional renormalization which is standard in dealing with free Gibbs laws and the associated matrix models.  The entropic quantity of interest in the large-$n$ limit will be $(1/n^2) h(\mu \mid \gamma)$ where $\gamma$ is the multi-GUE measure \eqref{eq: GUE tuple measure}.  The densities of interest are given by $\exp(-n^2 f)$ where $f$ is a function expressed in a dimension-independent way using the normalized trace.  We therefore use the following normalized version of pressure:
	\[
	\mathcal{P}(f) := -\frac{1}{n^2} \log \int_{(\mathbb{M}_n)_{\sa}^m} \exp(-n^2 f)\,d\gamma,
	\]
	for all $f: (\mathbb{M}_n)_{\sa}^m \to \R$ such that the integral converges.  We remark that straightforward and standard arguments show that if $f$ is Lipschitz, then the integral is finite due to the $\exp(-n^2 \norm{x}_2^2/2)$ term in the density of $\gamma$.
	
	We rely on the variational principles that express entropy and pressure as ``Legendre transforms'' of each other.
	
	\begin{lemma}\label{lem: entropy minimizer}
		Let $f:(\mathbb{M}_n)_{\sa}^m\to\R$ be Lipschitz with respect to $\norm{\cdot}_1$.
		Then $\mu_f$ is the unique minimizer of
		\[
		\mu \mapsto \frac{1}{n^2}h(\mu \mid \gamma) + \int f\,d\mu
		\]
		over all Borel probability measures $\mu$ on $(\mathbb{M}_n)_{\sa}^m$
		such that $\int |f|\,d\mu<\infty$, and the minimum is equal to
		$\mathcal{P}(f)$.
	\end{lemma}
	
	\begin{proof}
		This is the classical Gibbs variational principle; see, for instance,
		the entropy duality in Gozlan-Léonard~\cite[Appendix~B, Proposition~B.1]{GozlanLeonard2010} or Lehec~\cite[Equation~(2)]{Lehec2013}.
		For completeness, since
		\[
		\log\frac{d\mu_f}{d\gamma}
		=
		-n^2f+n^2\mathcal{P}(f),
		\]
		we have, for every Borel probability measure $\mu$ satisfying
		$\int |f|\,d\mu<\infty$,
		\[
		\frac{1}{n^2}h(\mu \mid \gamma)+\int f\,d\mu
		=
		\mathcal{P}(f)+\frac{1}{n^2}h(\mu \mid \mu_f).
		\]
		Moreover, since $f$ has at most linear growth and $\gamma$ is Gaussian,
		$\int |f|\,d\mu_f<\infty$, so $\mu_f$ belongs to the minimization domain.
		The result follows because $h(\mu \mid \mu_f)\geq0$, with equality
		if and only if $\mu=\mu_f$.
	\end{proof}

	\begin{lemma}\label{lem: Legendre transform of pressure}
		For any Borel probability measure $\nu$,
		\[
		\frac{1}{n^2} h(\nu \mid \gamma) = \sup_{g \in  C_b((\mathbb{M}_n)_{\sa}^m)} (\mathcal{P}(g) - \int g\,d\nu).
		\]
		where $C_b((\mathbb{M}_n)_{\sa}^m)$ is the space of bounded continuous real-valued functions on $(\mathbb{M}_n)_{\sa}^m$.
	\end{lemma}
	\begin{proof}
		This follows from the classical duality between relative entropy
		and the logarithmic Laplace transform; see Gozlan-Léonard~\cite[Appendix~B, Proposition~B.1]{GozlanLeonard2010} or Lehec~\cite[Equation~(2)]{Lehec2013}.
	\end{proof}
	\begin{lemma}
		If $f \in C_b((\mathbb{M}_n)_{\sa}^m)$,
		\[
		\frac{1}{n^2} h(\mu_f \mid \gamma) = \mathcal{P}(f) - \int f\,d\mu_f = \sup_{g\in C_b((\mathbb{M}_n)_{\sa}^m)} \left[ \mathcal{P}(g) - \int g\,d\mu_f \right]
		\]
	\end{lemma}
	
	\begin{proof}
		Note that
		\[
		\frac{d\mu_f}{d\gamma} = \frac{\exp(-n^2 f)}{\int \exp(-n^2 f)\,d\gamma},  
		\]
		and so
		\begin{align*}
			\frac{1}{n^2} h(\mu_f \mid \gamma) &= \frac{1}{n^2} \int (-n^2 f - \log \int \exp(-n^2 f)\,d\gamma) \frac{d\mu_f}{d\gamma} \,d\gamma \\
			&= -\int f\,d\mu_f + \mathcal{P}(f).
		\end{align*}
		We also observe that for every $g\in C_b((\mathbb{M}_n)_{\sa}^m)$
		\[
		\exp(n^2 \mathcal{P}(f)) = \int \exp(-n^2 f)\,d\gamma = \exp(n^2 \mathcal{P}(g)) \int \exp(-n^2(f-g))\,d\mu_g.
		\]
		and
		\[
		\frac{d\mu_f}{d\mu_g} = \frac{\exp(-n^2(f-g))}{\int \exp(-n^2(f-g))\,d\mu_g},
		\]
		and so similar computations as for $\mu_f$ and $\gamma$ yield
		\[
		\frac{1}{n^2} h(\mu_f \mid \mu_g) = \mathcal{P}(f) - \mathcal{P}(g) - \int (f-g)\,d\mu_f.
		\]
		Since $h(\mu_f \mid \mu_g) \geq 0$, we get
		\[
		\mathcal{P}(g) - \int g\,d\mu_f \leq \mathcal{P}(f) - \int f\,d\mu_f
		\]
		and equality is attained at $g=f$, which proves the lemma.
	\end{proof}
	
	\section{Proof of the main result via stochastic optimization}
	
	The proof of Theorem \ref{thm: main} is based on a stochastic variational formula for the pressure.  We state this result for the random matrix context with appropriate dimensional renormalization.  Since many works use the unnormalized trace $\Tr_n$ instead of $\tr_n$ to define the inner product, we will describe our normalization here explicitly.  First recall that $\mathbb{M}_n$ with the normalized inner product $\ip{\cdot,\cdot}_2$ admits an orthonormal basis
	\begin{equation} \label{eq: ONB}
		\mathcal{B} = \{ \sqrt{n} E_{j,j} \}_{j=1}^n \cup \{ \sqrt{n/2} (E_{j,k} + E_{k,j}) \}_{j < k} \cup \{ \sqrt{n/2} (i E_{j,k} - i E_{k,j} ) \}_{j < k}.
	\end{equation}
	It is not hard to see that an element with the GUE distribution \eqref{eq: GUE tuple measure} for $m = 1$ can be given as $\frac{1}{n} \sum_{b \in \mathcal{B}} Y_b b$ where $Y_b$ are independent standard normal variables, i.e.\ $Z$ is $1/n$ times a standard Gaussian vector in the inner product space $((\mathbb{M}_n)_{\sa}, \ip{\cdot,\cdot}_2)$.  The corresponding normalizations with respect to $n$ apply for a Gaussian $m$-tuple \eqref{eq: GUE tuple measure} and the Brownian motion in $(\mathbb{M}_n)_{\sa}^m$.  For further discussion on normalization, see \cite[Appendix B]{ST2022fractional} and 
	\cite[\S 3.3]{JekelPi2024}.  Applying the dimensional renormalization to the formula of Bou{\'e}--Dupuis and Borell, we obtain the following (see \cite[Lemma 6.2]{GJNPmatrixcontrols}).
	
	\begin{lemma}[See {\cite{BoueDupuis1998variational,Borell2000}}] \label{lem: BD formula for matrices}
		Consider the augmented filtration $\mathcal{G}_t$ generated by a Hermitian Brownian motion $(Z_{t,j}^{(n)})_{t \in [0,1],j \in [m]}$.  Let $f: (\mathbb{M}_n)_{\sa}^m \to \R$ be bounded and Borel.  Then
		\begin{equation} \label{eq: matrix Borell formula}
			\mathcal{P}(f) = -\frac{1}{n^2} \log \int_{(\mathbb{M}_n)_{\sa}^m} \exp(-n^2f) \,d\gamma = \inf_{\alpha} \mathbb{E} \left[ f\left(Z_1 + \int_0^1 \alpha \right) + \frac{1}{2} \int_0^1 \norm{\alpha}_2^2 \,dt \right],
		\end{equation}
		where $Z_t = (Z_{t,1}, \dots, Z_{t,m})$ and $\alpha = (\alpha_1,\dots,\alpha_m): \Omega \times [0,1] \to (\mathbb{M}_n)_{\sa}^m$ is progressively measurable satisfying
		\[
		\mathbb{E}\int_0^1\norm{\alpha_t}_2^2\,dt<\infty.
		\]
	\end{lemma}
	
	To prove Theorem \ref{thm: main}, we will show that there exists a minimizer of \eqref{eq: matrix Borell formula} which also satisfies $\norm{\alpha_t}_\infty \leq L$.  We remark that, if we know a minimizer exists, we could then apply the argument of \cite[Lemma 5.8, Step 3]{Jekel2026FIG} which shows that any candidate for the minimizer can be truncated to produce another candidate that satisfies $\norm{\alpha_t}_\infty \leq L$.  The existence of a minimizer can be deduced from general criteria such as \cite[Ch.~2, Theorem 5.3]{YongZhou1999}, but as stated this result relies on changing the underlying probability space and filtration.  To avoid this issue, it is convenient to use a more explicit description of the minimizer based the solution to an associated PDE (see e.g.\ \cite[\S IV.4]{FlemingSoner2006}).
	
	Define a function $u: [0,1] \times (\mathbb{M}_n)_{\sa}^m \to \R$ by
	\begin{equation} \label{eq: definition of u}
		u(t,X) = -\frac{1}{n^2} \log \int_{(\mathbb{M}_n)_{\sa}^m} \exp(-n^2 f(X+Y)) d\gamma_{1-t}(Y),
	\end{equation}
	where $\gamma_t$ is the $t^{1/2}$ rescaling of $\gamma$.  It is well known that $\exp(-n^2 u)$ solves the heat equation backward in time for an appropriately normalized Laplacian.  Namely, letting $\Delta^{(n)} = (1/n^2) \Delta$, we have
	\begin{equation} \label{eq: heat equation}
		-\partial_t (\exp(-n^2 u)) = \frac{1}{2} \Delta^{(n)} (\exp(-n^2 u)).
	\end{equation}
	(Note that the definition of $\Delta$ here is based on the normalized inner product $\ip{\cdot,\cdot}_2$ or the orthonormal basis \eqref{eq: ONB}.)  Rewriting \eqref{eq: heat equation} in terms of $u$ (i.e.\ the Hopf--Cole transformation) yields
	\begin{equation} \label{eq: HJB equation}
		-\partial_t u = \frac{1}{2} \Delta^{(n)} u - \frac{1}{2} \norm{\nabla u}_2^2, \qquad u(1,\cdot) = f.
	\end{equation}
	Moreover, Borell's formula can be used to express $u$ as the solution of a stochastic optimization problem (a special case of a much more general theory of Hamilton--Jacobi--Bellman equations):
	\begin{equation} \label{eq: matrix Borell general time}
		u(t,X) = \inf_{\alpha} \mathbb{E} \left[ f\left(X + Z_{t,1} + \int_t^1 \alpha \right) + \frac{1}{2} \int_t^1 \norm{\alpha}_2^2 \right],
	\end{equation}
	where $Z_{t,1} = Z_1 - Z_t$, and we use the augmented filtration on $[t,1]$ generated by the process $Z_{t,s} = Z_s - Z_t$ for $s \in [t,1]$ (i.e., the same setup as  in Lemma \ref{lem: BD formula for matrices} but starting time at $t$ instead of zero).  Observe that if $f$ is $L$-Lipschitz with respect to $\norm{\cdot}_1$, then so is $u(t,\cdot)$ for any $t \in [0,1]$, since \eqref{eq: matrix Borell general time} expresses $u(t,\cdot)$ as the infimum of a family of $L$-Lipschitz functions of $X$ with respect to $\norm{\cdot}_1$.  Furthermore, if $f$ is $C^2$, then $u$ is $C^{1,2}([0,1] \times (\mathbb{M}_n)_{\sa}^m)$ by standard facts about the heat equation \eqref{eq: heat equation} or about Gaussian convolution.
	
	\begin{lemma} \label{lem: explicit minimizer}
		Let $f$ be bounded, $C^2$, and $L$-Lipschitz with respect to $\norm{\cdot}_1$, and let $u$ be given by \eqref{eq: definition of u}.  Let $\mathcal{F}_t$ be the augmented filtration generated by the Brownian motion $Z_t$ on $[0,1]$.  For any $X_0 \in (\mathbb{M}_n)_{\sa}^m$ there exists a unique solution to the stochastic integral equation
		\[
		X_t = X_0 + Z_t - \int_0^t \nabla u(s,X_s)\,ds.
		\]
		Moreover, when $X_0 = 0$, the process $\alpha_t = -\nabla u(t,X_t)$ achieves the minimum in \eqref{eq: matrix Borell formula}.
	\end{lemma}
	
	\begin{proof}
		We claim that in fact for any continuous function $C(t)$ on $[0,1]$, there exists a unique solution $X(t)$ to
		\begin{equation} \label{eq: abstract integral equation}
			X(t) = X_0 + C(t) - \int_0^t \nabla u(s,X(s))\,ds.
		\end{equation}
		For every $s\in[0,1)$, every admissible control $\alpha$, and every
		$X,Y\in(\mathbb{M}_n)_{\sa}^m$, the Lipschitz assumption on $f$ gives
		\[\mathbb{E}\left[
		f\left(X+Z_{s,1}+\int_s^1\alpha_t\,dt\right)
		+\frac{1}{2}\int_s^1\norm{\alpha_t}_2^2\,dt
		\right]\leq
		\mathbb{E}\left[
		f\left(Y+Z_{s,1}+\int_s^1\alpha_t\,dt\right)
		+\frac{1}{2}\int_s^1\norm{\alpha_t}_2^2\,dt
		\right]
		+L\norm{X-Y}_1.
		\]
		Since the class of admissible controls is independent of the initial
		point, taking the infimum over $\alpha$ on both sides and applying
		\eqref{eq: matrix Borell general time} yields
		\[
		u(s,X)\leq u(s,Y)+L\norm{X-Y}_1.
		\]
		By symmetry, exchanging $X$ and $Y$, we obtain
		\[
		|u(s,X)-u(s,Y)|\leq L\norm{X-Y}_1.
		\]
		The same inequality holds for $s=1$, since $u(1,\cdot)=f$.
		Therefore $u(s,\cdot)$ is $L$-Lipschitz with respect to $\norm{\cdot}_1$ for every $s\in[0,1]$, Lemma~\ref{lem: Lipschitz gradient} implies that
		\[
		\norm{\nabla u(s,X)}_\infty\leq L
		\]
		for every $s\in[0,1]$ and $X\in(\mathbb{M}_n)_{\sa}^m$. In particular,
		$\norm{\nabla u(s,X_s)}_\infty\leq L$.  Let
		\[
		M = \norm{X_0}_\infty + \sup_{t \in [0,1]} \norm{C(t)}_\infty + L.
		\]
		For any continuous $F: [0,1] \to (\mathbb{M}_n)_{\sa}^m$, let
		\[
		\mathcal{I}(F)(t) = X_0 + C(t) - \int_0^t \nabla u(s,F(s))\,ds.
		\]
		Then $\norm{\mathcal{I}(F)(t)}_\infty \leq M$ for $t \in [0,1]$.  Thus, the $M$-ball in $C([0,1],(\mathbb{M}_n)_{\sa}^m)$ is invariant under $\mathcal{I}$.  Since $u$ is smooth, $\nabla u$ is Lipschitz on the $M$-ball of $\mathbb{M}_n)_{\sa}^m$ with some Lipschitz constant $L_M'$, and we have
		\[
		\norm{\mathcal{I}(F)(t) - \mathcal{I}(G)(t)}_\infty \leq L_M' \int_0^t \norm{F(s) - G(s)}_\infty \,ds.
		\]
		Thus, the solution to \eqref{eq: abstract integral equation} can be constructed by Picard iteration as usual.  The rate of convergence of the Picard iterates also depends only on $L_M'$.  Hence, it is not hard to show that the solution $X$ depends continuously on $C$ in $C([0,1],(\mathbb{M}_n)_{\sa}^m)$, thus producing a continuous map $\mathcal{X}: C([0,1],(\mathbb{M}_n)_{\sa}^m) \to C([0,1],(\mathbb{M}_n)_{\sa}^m)$ (which also depends on $X_0$).  Since Brownian motion is a random variable taking values in $C([0,1],(\mathbb{M}_n)_{\sa}^m)$, we can apply the continuous function $\mathcal{X}$ to it and obtain another random variable in $C([0,1],(\mathbb{M}_n)_{\sa}^m)$ which is the asserted solution to the integral equation.  It is also evident $\mathcal{I}(F)(t)$ only depends on $F|_{[0,t]}$, hence the same is true of $\mathcal{X}(F)(t)$, whence we see that $\mathcal{X}(Z)$ is progressively measurable.
		
		Now taking $X_0 = 0$ and letting $X_t = \mathcal{X}(Z)(t)$ and $\alpha_t = -\nabla u(t,X_t)$, we see that $\alpha_t$ is progressively measurable and
		\[
		Z_1^{(n)} + \int_0^1 \alpha_t\,dt = X_1.
		\]
		Moreover, the stochastic integral equation can be written in differential form as
		\[
		dX_t = dZ_t^{(n)} - \nabla u(t,X_t)\,dt.
		\]
		By the It{\^o} formula,
		\[
		d u(t,X_t) = \left[ \partial_t u(t,X_t) - \norm{\nabla u(t,X_t)}_2^2 + \frac{1}{2} \Delta^{(n)} u(t,X_t) \right]\,dt + \ip{\nabla u(t,X_t), dZ_t^{(n)}}_2.
		\]
		Applying the PDE \eqref{eq: HJB equation}, we obtain
		\[
		du(t,X_t) = \ip{\nabla u(t,X_t), dZ_t^{(n)}}_2 - \frac{1}{2} \norm{\nabla u(t,X_t)}_2^2\,dt.
		\]
		Now taking the
		expectation and integrating from $0$ to $1$ yields
		\[
		\mathbb{E} u(1,X_1) - u(0,0)
		=
		-\frac{1}{2}\mathbb{E}\int_0^1
		\norm{\nabla u(t,X_t)}_2^2\,dt,
		\]
		or
		\[
		u(0,0)
		=
		\mathbb{E}\left[
		f(X_1)
		+
		\frac{1}{2}\int_0^1
		\norm{\nabla u(t,X_t)}_2^2\,dt
		\right].
		\]
		Recall that by construction $u(0,0)$ is the left-hand side of \eqref{eq: matrix Borell formula}.  Thus, we have shown that $\alpha_t = -\nabla u(t,X_t)$ achieves the infimum on the right-hand side of \eqref{eq: matrix Borell formula}.  Moreover, $\norm{\alpha_t}_\infty \leq L$ since $\norm{\nabla u}_\infty \leq L$.
	\end{proof}
	
	Next, we prove Theorem \ref{thm: main} in the case that $f$ is $C^2$.  Afterward, we will handle the general case by an approximation argument.
	
	\begin{theorem} \label{thm: smooth case}
		Assume that $f$ is $C^2$ and bounded.  Let $\alpha$ be a minimizer on the right-hand side of \eqref{eq: matrix Borell formula} with $\norm{\alpha}_\infty \leq L$.  Let $X_t = Z_t + \int_0^t \alpha$.  Then $X_1 \sim \mu_f$ and $\norm{Z_1 - X_1}_\infty \leq L$.
	\end{theorem}
	
	\begin{proof}
		Let $\nu$ be the distribution of $X_1$.  By Lemma \ref{lem: entropy minimizer}, it suffices to show that $\nu$ minimizes $\mu \mapsto h(\mu \mid \gamma) - \int f\,d\mu$.  To this end, we first need to evaluate $h(\nu \mid \gamma)$.  For every bounded continuous function $g$ (this class suffices for entropy duality, as it is explicitly recalled at the beginning of~\cite[Section~2]{Lehec2013}), by definition of $\mathcal{P}(g)$,
		\[
		\mathcal{P}(g) \leq \mathbb{E}\left[ g(X_1) + \frac{1}{2} \int_0^1 \norm{\alpha}_2^2 \right] = \int g\,d\nu + \frac{1}{2} \int_0^1 \mathbb{E} \norm{\alpha}_2^2;
		\]
		moreover, equality holds in the case that $g = f$ since $\alpha$ was chosen to be a minimizer.  Therefore,
		\[
		\mathcal{P}(g) - \int g\,d\nu \leq \mathcal{P}(f) - \int f\,d\nu,
		\]
		so by Lemma \ref{lem: Legendre transform of pressure}
		\[
		\frac{1}{n^2} h(\nu \mid \gamma) = \mathcal{P}(f) - \int f\,d\nu = \frac{1}{2} \int_0^1 \mathbb{E} \norm{\alpha}_2^2.
		\]
		Now let $\mu$ be another measure.  By Lemma \ref{lem: Legendre transform of pressure},
		\[
		\frac{1}{n^2} h(\mu \mid \gamma) + \int f\,d\mu \geq \mathcal{P}(f) = \frac{1}{n^2} h(\nu \mid \gamma) + \int f\,d\nu.
		\]
		Therefore, $\nu$ is a minimizer of $\mu \mapsto h(\mu \mid \gamma) + \int f\,d\mu$, so by Lemma \ref{lem: entropy minimizer}, we have $\nu = \mu_f$.
		
		Hence, $X_1 \sim \mu_f$ as asserted.  Moreover,
		\[
		\norm{X_1 - Z_1}_\infty = \norm*{\int_0^1 \alpha }_\infty \leq \int_0^1 \norm{\alpha}_\infty \leq L.  \qedhere
		\]
	\end{proof}
	
	\begin{proof}[Conclusion of the proof of Theorem \ref{thm: main}]
		Consider a general $f: (\mathbb{M}_n)_{\sa}^m \to \R$ that is $L$-Lipschitz with respect to $\norm{\cdot}_1$.  For $\delta > 0$, let $\theta_\delta(s)= \max\{-\delta^{-1},\min\{s,\delta^{-1}\}\}$ for $s\in \mathbb R$
		and let
		\[
		f_\delta(x) = \int_{(\mathbb{M}_n)_{\sa}^m} \theta_\delta\bigl(f(x-\delta^{1/2}y)\bigr)\,d \gamma(y)
		\]
		be the approximation of $f$ obtained by solving the heat equation or by Gaussian convolution, after truncation using $\theta_\delta$.  Set $f_0=f$.  Observe the following:
		\begin{enumerate}
			\item $f_\delta\to f$ locally uniformly as $\delta\to 0^+$.
			\item $f_\delta$ is $L$-Lipschitz with respect to $\norm{\cdot}_1$.
			\item $f_\delta$ is bounded and\ $C^\infty$ and its derivatives of all orders are bounded for each fixed $\delta>0$.
			\item There are some constants $A$ and $B$ such that $|f_\delta(x)| \leq A + B \norm{x}_1$ for all $\delta \in [0,1]$.
			\item Suppose $\varphi \in C((\mathbb{M}_n)_{\sa}^m)$ with $|\varphi(x)| \leq C \exp(c \norm{x}_1)$ for some constants $c$ and $C$.  Then $|\varphi| \exp(-n^2 f_\delta)$ is also bounded by some function of the form $C' \exp(c' \norm{x}_1)$ uniformly for $\delta \in [0,1]$, hence by the dominated convergence theorem
			\[
			\lim_{\delta \to 0^+} \int_{(\mathbb{M}_n)_{\sa}^m} \varphi \exp(-n^2 f_\delta)\,d\gamma = \int_{(\mathbb{M}_n)_{\sa}^m} \varphi \exp(-n^2 f)\,d\gamma.
			\]
			\item In particular, $\mu_{f_\delta} \to \mu_f$ in the weak-$*$ topology as $\delta \to 0^+$.
		\end{enumerate}
		Now apply Theorem \ref{thm: smooth case} to $\mu_{f_\delta}$ to obtain random variables $Z \sim \gamma$ and $X_\delta \sim \mu_{f_\delta}$ such that $\norm{Z - X_\delta}_\infty \leq L$ almost surely.  Let $\mathcal{S}$ be the set of probability measures given as the distribution of $(Y,Z)$ where $Z \sim \gamma$ and $\norm{Y - Z}_\infty \leq L$.  Observe that $\mathcal{S}$ is tight, i.e.\ weak-$*$ compact and hence relatively weak-$*$ compact; for instance, one can use the map $(Y,Z) \mapsto (Y-Z,Z)$ to transform $\mathcal{S}$ into the set of measures whose second marginal is $\gamma$ and whose first marginal is supported in $\{x: \norm{x}_\infty \leq L\}$.  Therefore, there exists a sequence $\delta_k \to 0$ such that $(X_{\delta_k},Z)$ converges in distribution as $k \to \infty$.  Let $(X,Z)$ be a pair realizing the limiting distribution.  Then $Z \sim \gamma$ and $X \sim \mu_f$ since $X_{\delta_k} \sim \mu_{f_{\delta_k}} \to \mu_f$, and $\norm{X-Z}_\infty\leq L$ almost surely by the Portmanteau theorem, since
		$\{(y,z):\norm{y-z}_\infty\leq L\}$ is closed.
	\end{proof}
	
	\begin{remark}
		There is a version of Theorem \ref{thm: main} for non-self-adjoint matrices as well.  Let $f: \mathbb{M}_n^m \to \R$ be $L$-Lipschitz with respect to $\norm{\cdot}_1$, and let
		\[
		d\mu_f(x) = \text{\emph{(constant)}} e^{-n^2(\norm{x}_2^2/2 + f(x))}\,dx
		\]
		on $\mathbb{M}_n^m$.  Let $\tilde{\gamma} = \mu_0$ be the normalized Gaussian measure on $\mathbb{M}_n^m$ (so each coordinate is a Gaussian Ginibre rather than GUE random matrix, or the operator real and imaginary parts are GUE).  Then there exist random variables $X \sim \mu_f$ and $Z \sim \tilde{\gamma}$ with $\norm{X - Z}_\infty \leq L$.  The proof is exactly the same as that of Theorem \ref{thm: main}.  Here the gradient for $f$ on $\mathbb{M}_n^m$ is understood with respect to the real inner product $\re \ip{x,y}_2$.
	\end{remark}
	
	\begin{remark} \label{rem: alternative proof}
		It is worthwhile noting that an alternative proof can be obtained by adapting the synchronous coupling argument for Langevin diffusion of Khudiakova,
		Maas, and Pedrotti \cite[Theorem~2.1, Step~1]{KhudiakovaMaasPedrotti2025}.  For general background on Markov diffusion process, see e.g.\ \cite{BGL2014}.  For smooth $f$ such that $\nabla f$ is globally Lipschitz with respect to $\norm{\cdot}_2$, let $({Z}_t)_{t\geq 0}$ be a $m$-tuple of independent GUE Brownian motions normalized so that $Z_1^{(n)}\sim \gamma$ 
		\begin{align}
			dX_t &= -(X_t+\nabla f(X_t))\,dt+\sqrt{2}\,d{Z}_t,\\
			dY_t &= -Y_t\,dt+\sqrt{2}\,d{Z}_t,
		\end{align}
		with $X_0=Y_0\sim\mu_f$, independent of the Brownian motion.
		Then $X_t\sim\mu_f$ for every $t\geq0$. Indeed, boundedness of the second derivatives of $f$ implies that $\nabla f$ is globally Lipschitz in $\norm{\cdot}_2$. Hence the drifts of
		both equations are globally Lipschitz for the $\norm{\cdot}_2$, so standard
		results for SDEs give unique strong solutions for all $t\geq0$.
		Moreover, integration by parts shows that $\mu_f$ is invariant
		for the first diffusion, so the choice $X_0\sim\mu_f$ makes
		$(X_t)_{t\geq0}$ stationary. On the other hand, the distribution
		of $Y_t$ converges weakly to $\gamma$. In the case considered
		here, the difference satisfies
		\[
		d(X_t - Y_t) = -(X_t - Y_t)\,dt - \nabla f(X_s)\,dt
		\]
		as the Brownian terms cancel, and hence by the standard integrating-factor computation,
		\begin{equation} \label{eq: exponential integrating factor}
			X_t-Y_t=-\int_0^t e^{-(t-s)}\nabla f(X_s)\,ds,
		\end{equation}
		Since $f$ is $L$-Lipschitz with respect to $\norm{\cdot}_1$,
		Lemma~\ref{lem: Lipschitz gradient} gives
		$\norm{\nabla f(x)}_\infty\leq L$. Hence
		$\norm{X_t-Y_t}_\infty\leq L(1-e^{-t})$ almost surely.
		Letting $t\to\infty$ and using lower semi-continuity of the $L^{\infty}$-Wasserstein distance
		$W_\infty$, we obtain $W_\infty(\mu_f,\gamma)\leq L$.
		The general case follows by approximation as in our original proof of Theorem \ref{thm: main}.
		
		Furthermore, this method also establishes a more general estimate.  Namely, \eqref{eq: exponential integrating factor} and the fact that $X_s \sim \mu_f$ implies that for any $p, q \in [1,\infty]$, we have
		\begin{equation}
			\norm{\norm{X_t - Y_t}_p}_{L^q(\Omega,\mathbb{P})} \leq \int_0^t e^{-(t-s)} \norm{\norm{\nabla f(X_s)}_p}_{L^q(\Omega,\mathbb{P})} \,ds = (1 - e^{-t}) \norm{\norm{\nabla f}_p}_{L^q(\mu_f)},
		\end{equation}
		where $(\Omega,\mathbb{P})$ denotes the underlying probability space.  Then by approximation arguments, one can obtain a coupling $(X,Z)$ of $\mu_f$ and $\gamma$ with $\norm{\norm{\nabla f}_p}_{L^q(\mu_f)}$.  This provides an $L^q$-version of the log-Sobolev-plus-Talagrand inequality with respect to the norm $\norm{\cdot}_p$ on $(\mathbb{M}_n)_{\sa}^m$.
		
		Note that the same proof can be applied to a measures on Euclidean space $\R^m$ of the form $e^{-f(x) - |x|^2/2}\,dx$, with the $\norm{\cdot}_p$ replaced by \emph{any} choice of norm $\norm{\cdot}$ on $\R^m$, and here the Lipschitz hypothesis for $f$ is assumed with respect to the dual norm
		\[
		\norm{x}' = \sup \{ |\ip{x,y}|: \norm{y} \leq 1 \}.
		\]
		In other words, the $L^\infty$ and $L^p$ bounds in \cite[Theorem 2.1, Theorem A.1]{KhudiakovaMaasPedrotti2025} can be generalized to allow any choice of norm on $\R^m$, provided that the base measure is standard Gaussian.
	\end{remark}
	
	\subsection*{Funding}
	
	C-P. Diez acknowledges support from ISF grant 2574/24 and NSF--BSF grant 2022707.
	
	\bibliographystyle{plain}
	\bibliography{references}
	
\end{document}